\documentclass{amsart}

\usepackage{algorithm}
\usepackage{algpseudocode}
\usepackage{amsmath}
\usepackage{amsxtra}
\usepackage{amscd}
\usepackage{amsthm}
\usepackage{amsfonts}
\usepackage{amssymb}
\usepackage{mathtools}
\usepackage{eucal}
\usepackage[all]{xy}
\usepackage{graphicx}
\usepackage{mathrsfs}
\usepackage{color}
\usepackage{caption}
\usepackage{enumerate}
\usepackage{faktor}
\usepackage{bbm}
\usepackage{hyperref}
\usepackage[T1]{fontenc} % improved font encoding
\usepackage[utf8]{inputenc} % for better handling of non-ASCII characters
\usepackage{tikz}
\usepackage{tikz-cd}
\usetikzlibrary{matrix,arrows,decorations.pathmorphing}

\hypersetup{
  colorlinks = true,
  urlcolor = cyan,
}

\theoremstyle{plain}
\newtheorem{theorem}{Theorem}[section]

\newtheorem{proposition}[theorem]{Proposition}

\theoremstyle{definition}
\newtheorem{definition}[theorem]{Definition}

\theoremstyle{remark}

\usepackage{placeins}

\numberwithin{equation}{section}

\newcommand{\C}{\mathbb{C}}

\newcommand{\N}{\mathbb{N}}

\newcommand{\Q}{\mathbb{Q}}
\newcommand{\R}{\mathbb{R}}

\newcommand{\calm}{\mathcal{M}}

\newcommand\restr[2]{{\left.\kern-\nulldelimiterspace#1\right|_{#2}}}

\makeatletter
\AtBeginDocument
 {
    \def\@thm#1#2#3{%
      \ifhmode
        \unskip\unskip\par
      \fi
      \normalfont
      \trivlist
      \let\thmheadnl\relax
      \let\thm@swap\@gobble
      \let\thm@indent\indent % indent
      \thm@headfont{\scshape}% heading font small caps
      \thm@notefont{\fontseries\mddefault\upshape}%
      \thm@headpunct{.}% add period after heading
      \thm@headsep 5\p@ plus\p@ minus\p@\relax
      \thm@space@setup
      #1% style overrides
      \@topsep \thm@preskip               % used by thm head
      \@topsepadd \thm@postskip           % used by \@endparenv
      \def\dth@counter{#2}%
      \ifx\@empty\dth@counter
        \def\@tempa{%
          \@oparg{\@begintheorem{#3}{}}[]%
        }%
      \else
        \H@refstepcounter{#2}%
        \hyper@makecurrent{#2}%
        \let\Hy@dth@currentHref\@currentHref
        \AddToHookNext{para/begin}{\MakeLinkTarget*{\Hy@dth@currentHref}}%
        \def\@tempa{%
          \@oparg{\@begintheorem{#3}{\csname the#2\endcsname}}[]%
        }%
      \fi
      \@tempa
    }%
  \dth@everypar={%
    \@minipagefalse \global\@newlistfalse
    \@noparitemfalse
    \if@inlabel
      \global\@inlabelfalse
      \begingroup \setbox\z@\lastbox
       \ifvoid\z@ \kern-\itemindent \fi
      \endgroup
      \unhbox\@labels
    \fi
    \if@nobreak \@nobreakfalse \clubpenalty\@M
    \else \clubpenalty\@clubpenalty \everypar{}%
    \fi
  }%
}

\title{$p$-adic alternated Julia sets}
\author{Rui-Xi Wang}
\subjclass[2020]{11S82, 37F10}
\keywords{Julia sets, $p$-adic dynamics, alternated Julia sets}
\subjclass[2020]{37F10, 11S82}

\address{Department of Mathematics, Massachusetts Institute of Technology, 77 Massachusetts Avenue
Cambridge, MA, USA}
\email{rxwangtw@mit.edu, a36466136@gmail.com}

\date{\today}

\begin{document}

\begin{abstract}
  % \normalsize  
  The study of dynamical systems involves analyzing how functions behave under iteration in different mathematical spaces. In the context of complex dynamics, tools such as Julia sets and filled Julia sets are used to understand the long-term behavior of functions in the complex Euclidean field. In this paper, we will present a review of Julia sets and filled Julia sets, provide an overview of the mathematical formulation of the alternated Julia sets introduced in the work of Danca-Romera-Pastor, extend it to the $p$-adic setting, and propose a tool that can potentially be used to study the arithmetic dynamics of various types of functions. Additionally, we will summarize key results on connectivity properties and visualization techniques as discussed in the work of Danca-Bourke-Romera and provide a visualization algorithm and pseudocode that enable the visualization of alternated Julia sets with various connectivity properties.
\end{abstract}

\maketitle

\setcounter{tocdepth}{2} %change the depth of the table of contents
\tableofcontents

%%%%%%%%%%%%%%%%%%%%%%%%%%%% Introduction %%%%%%%%%%%%%%%%%%%%%%%%%%%%
\section{Introduction}
In complex dynamics, iterated functions generate rich and intricate structures known as Julia sets. For a polynomial 
\( F: \C \to \C \), the \textbf{filled Julia set} is the set of points whose orbits under iteration of \( F \) remain bounded, while the \textbf{Julia set} is its boundary. 
Variants such as the \textbf{alternated Julia set}, which alternates between different polynomials at each step of iteration, reveal new fractal geometries and dynamical behaviors beyond the classical setting.

This paper introduces a new object of study: the \textbf{$p$-adic alternated Julia set}. Extending the concept of alternated Julia sets to the non-Archimedean setting of $\C_p$, the $p$-adic analog of the complex field $\C$, we define these sets analogously using the $p$-adic norm to assess orbital boundedness. The resulting dynamics exhibit features distinct from their complex counterparts, shaped by the topology of $\C_p$.

We begin by reviewing the classical theories and properties of Julia sets, filled Julia sets, Mandelbrot sets, and alternated Julia sets introduced in \cite[Section 2]{MR2567586}, along with a visualization algorithm and examples illustrating a range of connectivity properties of alternated Julia sets discussed in \cite[Sections 2 and 3]{MR3080737}. Building on this foundation, we review the existing theory of $p$-adic Mandelbrot and Julia sets before formulating a definition of $p$-adic alternated Julia sets by composing two given functions or dynamics in the $p$-adic setting. We then prove several foundational properties of these sets, highlighting the novel behaviors of alternated Julia sets in the $p$-adic context and drawing comparisons with the classical complex case. Specifically, the connectivity of a given $p$-adic alternated Julia set can be determined similarly to its complex counterpart—that is, by checking the boundedness of the auxiliary $p$-adic Julia set generated by composing the two alternated dynamics. The patterns of some $p$-adic alternated Julia sets can also be determined by comparing the magnitude of the degree of the composed dynamics and the $p$-value, as described in \cite[Theorem 1]{MR3113229}.

\subsection{Acknowledgements}
The author is deeply grateful to Robin Zhang for organizing MIT's 18.784 (Seminar in Number Theory) and for his invaluable guidance and feedback throughout the writing process. The author also thanks Alejandro Reyes and Lucy Epstein for their insightful discussions and detailed feedback during the 18.784 peer review session.

\section{Julia sets, filled Julia sets, and Mandelbrot sets}
Julia sets and filled Julia sets are common tools for studying the complex-dynamical behavior of a given function. For a given function $F:\C\rightarrow \C$, the filled Julia set $K$ is the set of points where their corresponding orbits under $F$ are bounded. In other words, we can define the filled Julia set as follows:
\begin{definition}
    The filled Julia set $K(F)$ of a given function $F:\C\rightarrow \C$ is the set of points $z$ where the orbit $\{F^n(z)\}$ is bounded.
\end{definition}
 With the definition of the filled Julia set $K(F)$, we can also define the Julia set $J(F)$ that can be used to describe the dynamical property of a given function.
 \begin{definition}
 The Julia set $J(F)$ is the boundary of the filled Julia set $K(F)$. In mathematical terms, $J(F) = \partial K(F)$.
\end{definition}
The points in the Julia set exhibit chaotic behaviors. More specifically, points lying outside the filled Julia set diverge to infinity under repeated iterations of the function $F$, while points in the interior of the filled Julia set remain bounded. In contrast, the dynamical behavior of points in the Julia set is highly sensitive to perturbations and lacks stability. 

Notably, the Julia set and the filled Julia set associated with a given function share the same connectivity properties. For a polynomial in the quadratic family $F:\C\rightarrow\C$, $F(z) = z^2 + c$, the corresponding Julia set and filled Julia set are either connected or totally-disconnected. The set of complex numbers $c$ that can form a connected Julia set is called the Mandelbrot set. Different components of the Mandelbrot set correspond to different dynamical behaviors of the corresponding filled Julia sets, resulting in different fractal patterns that can be used to generate beautiful visualizations. 
\subsection{Examples}
The visualization of the Mandelbrot set and the corresponding connected filled Julia sets of different points in the Mandelbrot set are provided in Figure~\ref{fig:Julia_Mandelbrot}, adapted from \cite[Figure 1]{bourke2001juliaset}. Different components of the Mandelbrot set correspond to different dynamical behaviors of the corresponding filled Julia sets. For example, the largest bulb of the Mandelbrot set corresponds to the filled Julia sets that have attractive fixed points (i.e., there exists an $x$ such that $|F'(x)| < 1$), while the second-largest bulb corresponds to the filled Julia sets that have attractive 2-cycles (i.e., there exists an $x$ such that $|(F^2)'(x)| < 1$). 
\begin{figure}[htbp]
  \centering
  \includegraphics[width=0.5\textwidth]{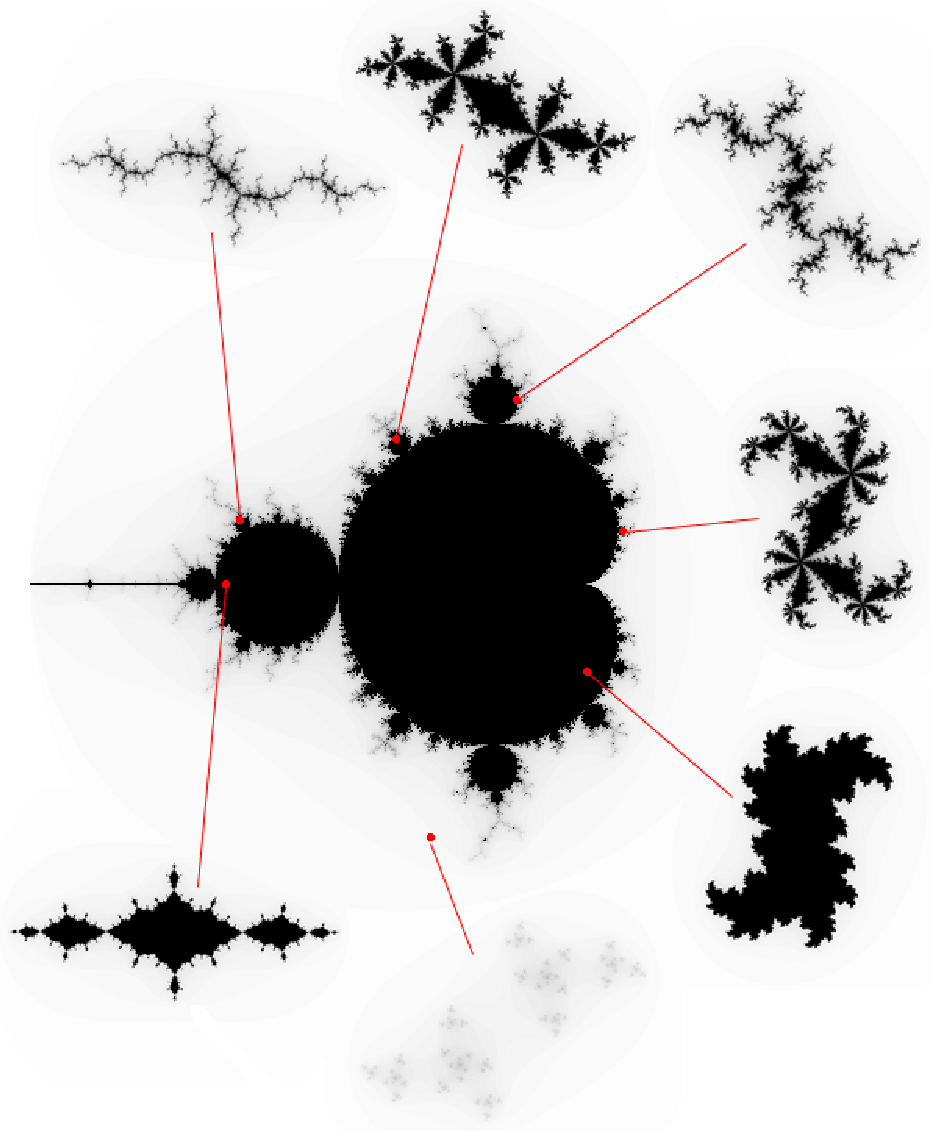} % Replace with your image file
  \caption{A visualization of the Mandelbrot set and the corresponding connected Julia sets of different components of the Mandelbrot set. The plot is generated by Python and is adapted from \cite[Figure 1]{bourke2001juliaset}.}
  \label{fig:Julia_Mandelbrot}
\end{figure}
\FloatBarrier
\section{Alternated filled Julia sets}
\subsection{Formulations and Boundedness Properties} 
 Although we can use Julia sets to model the behavior of a given dynamical system, there are multiple types of interactions in some more complex systems that cannot be described with a single function. To provide a more general framework for the problem, imagine that there are two different functions in the quadratic family, $F_1$ and $F_2$, that act on a given point $z_0$ iteratively. We can define the corresponding orbit of iteratively applying the two given functions in the quadratic family as follows.
\begin{definition}
    For two given functions 
    % \[
    %   \begin{tikzcd}
    %     F_1: \C \arrow[r] & \C \\
    %     z \arrow[r,mapsto] & z^2 + c_1 \\
    %     F_2: \C \arrow[r] & \C \\
    %     z \arrow[r,mapsto] & z^2 + c_2 
    %   \end{tikzcd}
    % \]
        \[
      \begin{tikzcd}[row sep=tiny]
        F_1: \C \arrow[r] & \C \\
        z \arrow[r, mapsto] & z^2 + c_1 \end{tikzcd}\]
         \text{ and}
        \[\begin{tikzcd}[row sep=tiny]
        F_2: \C \arrow[r] & \C \\
        z \arrow[r, mapsto] & z^2 + c_2 
      \end{tikzcd}\]
     where $c_1$ and $c_2$ are constants in $\C$, we define the orbit $P_{c_1c_2}(z_0)$ of a given starting point $z_0\in\C$ as the set \begin{align*}
         \{z_i|z_{2i} = F_2(z_{2i-1}) = z_{2i-1}^2+c_2, z_{2i+1} = F_1(z_{2i})=z_{2i}^2+c_1\}
     \end{align*} where the even terms of the orbits are generated by applying $F_2$ on the previous terms while the odd terms are generated by applying $F_1$ on the previous terms.
\end{definition}
Similarly, we can also swap $F_1$ and $F_2$ to define another orbit $P_{c_2c_1}(z_0)$:   
\begin{align*}
        P_{c_2c_1}(z_0) = \{z_i|z_{2i} = F_1(z_{2i-1}) = z_{2i-1}^2+c_1, z_{2i+1} = F_2(z_{2i})=z_{2i}^2+c_2\}.
    \end{align*}
  Similar to the definition of filled Julia sets, we can also define the alternated filled Julia set of two alternated dynamics by the boundedness of the orbits.
  \begin{definition}
      The alternated filled Julia set $K_{c_1c_2}$ corresponding to the orbit type $P_{c_1c_2}$ is defined as the set of points where their orbits under the alternated dynamics are bounded:
        \[
        K_{c_1c_2} = \{z|P_{c_1c_2}(z)\text{ is bounded}\}.
        \] 
  \end{definition}
  To analyze whether a given orbit $P_{c_1c_2}(z)$ is bounded, we need the following proposition that relates the boundedness of the odd terms and the even terms of the orbit.
  \begin{proposition}\cite[Proposition 2.1]{MR2567586}
    \label{prop:boundedness}
    \begin{enumerate}
        \item If $z_{2i}$ is bounded, then $z_{2i-1}$ is also bounded.
        \item If $z_{2i}$ is unbounded, $z_{2i+1}$ is also unbounded.
    \end{enumerate}
  \end{proposition}
  \begin{proof}
      $z_{2i-1} = \sqrt{z_{2i}-c_2}$ is bounded if $z_{2i}$ is also bounded. \\Similarly, $z_{2i+1} = z_{2i}^2 + c_1$ is unbounded if $z_{2i}$ is also unbounded.
  \end{proof}
  We can formulate an auxiliary orbit (which can be used to study the boundedness of the alternated orbit) of the alternated dynamics for a given starting point $z_0\in\C$ as follows. 
  \begin{definition}\cite[Equation 5]{MR2567586}:
      A quartic auxiliary orbit $Q_{c_1c_2}(z_0)$ of a given starting point $z_0$ and given parameters $c_1, c_2\in\C$ can be defined as the even terms of the alternated orbit $P_{c_1c_2}(z_0)$:
        \[
      Q_{c_1c_2}(z_0) = \{z_i|z_i=(z_{i-1}^2+c_1)^2+c_2\}. 
      \]
  \end{definition}
  With the assistance of Proposition \ref{prop:boundedness}, we can have the following theorem.
  \begin{theorem} \cite[Theorem 2.2 and 2.3]{MR2567586}
        \label{thm:qua}
      The boundedness of $Q_{c_1c_2}(z_0)$ implies the boundedness of $P_{c_1c_2}$ and therefore the quartic system \[F(z) = (z^2+c_1)^2 + c_2\] and the alternated system share the same filled Julia set.
  \end{theorem}
  This gives us a tool to analyze the boundedness and connectivity of a given alternated filled Julia set. \\
  \subsection{Connectivities} 
  As stated in the previous section, every alternated filled Julia set has an equivalent auxiliary quartic filled Julia set. We can hence investigate the connectivity of a given filled Julia set by analyzing the connectivity of the auxiliary filled Julia set. To explore the connectivity of a given (filled) Julia set, we need a theorem that was proven by Scott Sutherland. 
  \begin{theorem}\cite[Section 6]{sutherland2014}
  \label{thm:connectivity}
\begin{enumerate}
\item The Julia set $J(F)$ of a function $F:\C\rightarrow\C$ is connected if and only if the orbits of all critical points of $F$ under iteration are bounded.
\item The Julia set $J(F)$ of a function 
$F:\C\rightarrow\C$ is totally disconnected if and only if the orbits of all critical points of 
$F$ under iteration are unbounded. Julia sets of this type are commonly referred to as Fatou dust.
\end{enumerate}
  \end{theorem}
  As we have shown in Theorem \ref{thm:qua}, the Julia set of the quartic system $F(z) = (z^2+c_1)^2+c_2$ is equivalent to the alternated Julia set. We can analyze the connectivity of the alternated filled Julia set by analyzing $F$. By solving the following equation 
  \[
  F'(z) = 2(z^2+c_1)*2z = 0,
  \]
  we can get three critical points of F to be $0$ and $\pm\sqrt{-c_1}$. By discussing the properties of the corresponding critical orbits, we can determine the connectivity of the alternated Julia set we are interested in.
  \begin{theorem} \cite[Section 3.1]{MR2567586}
      \begin{enumerate}
          \item The alternated Julia set is connected if the orbits of $0$ and $\pm\sqrt{-c_1}$ are bounded.
          \item 
          The alternated Julia set is totally disconnected if the orbits of $0$ and $\pm\sqrt{-c_1}$ are unbounded.
          \item The alternated Julia set is disconnected if the orbit of either $0$ or $\pm\sqrt{-c_1}$ is unbounded.
      \end{enumerate}
  \end{theorem}
  As the result, the boundedness of a given alternated Julia set can be inferred through simulating the alternated dynamics in a computer program. 
  \subsection{Visualization algorithms and examples}
  In this section, we move from the theoretical formulations of the alternated Julia set to its visual representation. By leveraging computer programs, we can transform the mathematical concepts into vivid and intuitive images, offering a new perspective on the structure and complexity of the set.
  Since the behaviors of a filled Julia set can be modeled by the corresponding quartic function, it seems like the visualizations of the desired set can be easily shown with the JuliaSetPlot command in Mathematica. However, due to the lack of computational power, the command cannot be directly plotted in an online Mathematica notebook. As a result, we decide to provide a visualization algorithm that can be run in a Google Colab notebook. To simulate the dynamics, we apply the alternated functions to a set of initial points within a given range (in our case, we choose $a+bi$ for $a, b\in\R$, $|a| < 1.5$, $|b| < 1.5$). We set a boundary of absolute distances between $F^n(z_0)$ and the origin for a given $z_0$. After that, we plot the number of iterations needed for a given point to go beyond the threshold. The pseudocode for the visualization algorithm can be found in Algorithm \ref{alg:pseudocode}. We choose three different pairs of $c_1$ and $c_2$, representing different types of connectivities of the corresponding alternated Julia sets. The visualization can be found in Figure \ref{fig:totally_disconnected}, \ref{fig:disconnected}, and \ref{fig:connected}. \\
  \begin{algorithm}
\caption{Alternated Julia sets visualization}
\label{alg:pseudocode}
\begin{algorithmic}
\State \textbf{Input:} \textit{width}, \textit{height}, \textit{zoom}, \textit{center\_x}, \textit{center\_y}, \textit{$c_1$}, \textit{$c_2$}, \textit{max\_iter}, \textit{threshold}
\State \textbf{Output:} \textit{image}[height][width]

\For{$y = 0$ to $height - 1$}
    \For{$x = 0$ to $width - 1$}
        \State $zx \gets (x - \frac{width}{2}) \cdot zoom + center\_x$
        \State $zy \gets (y - \frac{height}{2}) \cdot zoom + center\_y$
        \State $z \gets zx + i \cdot zy$
        \State $iter \gets 0$
        \While{$iter < max\_iter$ \textbf{and} $|z| < threshold$} 
            \If{$iter$ is odd}
                \State $z \gets z^{2} + c_1$
            \Else
                \State $z \gets z^{2} + c_2$
            \EndIf   
            \State $iter \gets iter + 1$
        \EndWhile
        \State $image[y][x] \gets iter$
    \EndFor
\EndFor

\State \Return $image$
\end{algorithmic}
\end{algorithm}

\begin{figure}[htbp]
  \centering
  \includegraphics[width=0.5\textwidth]{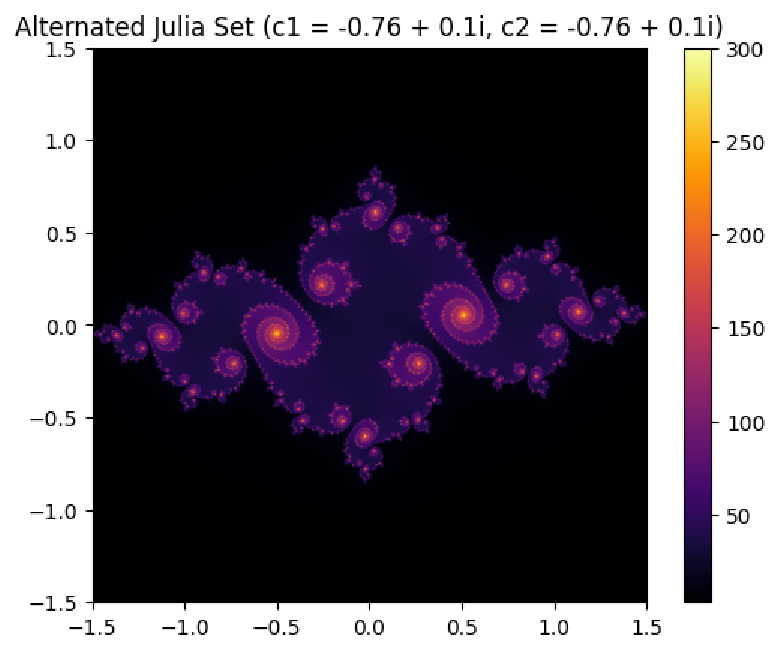} % Replace with your image file
  \caption{A totally disconnected alternated Julia set with $c_1=-0.76+0.1i$, $c_2=-0.76+0.1i$. }
  \label{fig:totally_disconnected}
\end{figure}

\begin{figure}[htbp]
  \centering
  \includegraphics[width=0.5\textwidth]{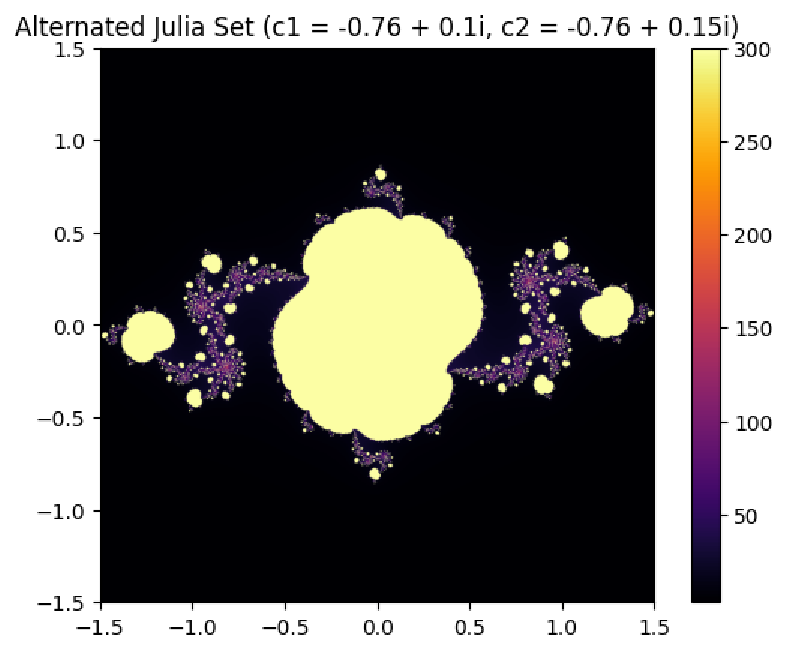} % Replace with your image file
  \caption{A disconnected alternated Julia set with $c_1=-0.76+0.1i$, $c_2=-0.76+0.15i$. }
  \label{fig:disconnected}
\end{figure}

\begin{figure}[htbp]
  \centering
  \includegraphics[width=0.5\textwidth]{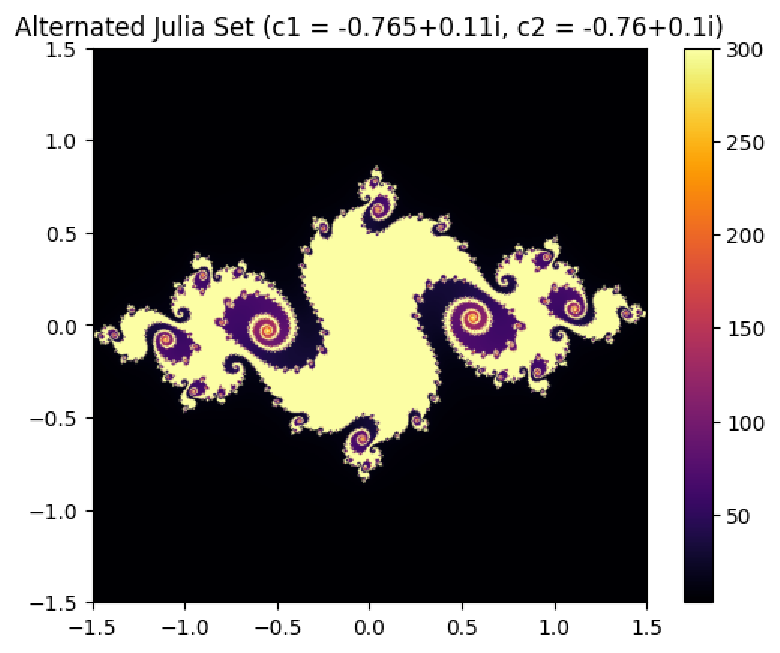} % Replace with your image file
  \caption{A connected alternated Julia set with $c_1=-0.765+0.11i$, $c_2=-0.76+0.1i$. }
  \label{fig:connected}
\end{figure}
\FloatBarrier
\section{$p$-adic Mandelbrot sets and $p$-adic Julia sets} Having concluded our examination of the alternated Julia set within the framework of complex dynamics, we now proceed to consider analogous constructions in the context of $p$-adic Mandelbrot and Julia sets. This transition from the complex to the $p$-adic setting highlights profound differences in topology and analytic behavior, offering new perspectives on dynamical systems.
In the complex dynamical setting, the Mandelbrot set is defined by the connectivities of Julia sets of functions $F(z) = z^2 + c$ in the quadratic family. We can therefore formulate a definition of $p$-adic Mandelbrot sets of quadratic families with boundedness of orbits of the origin. To begin with, we need to define the $p$-adic analog of the complex field $\C$.
\begin{definition}
    The $p$-adic analog $\C_p$ of the complex field $\C$ is defined as
\[
\C_p = \widehat{\overline{\Q_p}}
\]
where \(\overline{\Q_p}\) is the algebraic closure of the field of \(p\)-adic numbers \(\Q_p\),
    and \(\widehat{\overline{\Q_p}}\) denotes the completion of \(\overline{\Q_p}\) with respect to the \(p\)-adic norm \(|\cdot|_p\).
\end{definition}
We can then define the $p$-adic Mandelbrot set in the field.
\begin{definition}
    \label{def:order}
    The $p$-adic Mandelbrot set $\calm_p^2$ for quadratic families in $\C_p$ consists of the points where the sequence $\{F^n(0)\}$, generated by iterating the function \[F: \C_p \to \C_p, \text{ } F(z) = z^2 + c \text{ for a given } c\in\C_p\] remains bounded.
\end{definition}
However, the structure of $\calm_p^2$ that corresponds to the quadratic family is simple in contrast to its complex-dynamical counterpart. 
\begin{theorem}\cite[Page 2]{MR3113229}
    $\calm_p^2$ is the unit disk around the origin. 
\end{theorem}
\begin{proof}
    For the case $|c|_p \leq 1$, $|F(0)|_p = |c|_p \leq 1$.
Suppose that for all $n\leq N\in\N$, there is a positive number $r$ such that $|c|_p \leq r \leq 1$ and $|F^n(0)|_p \leq r$. Then, \[|F^{n+1}(0)|_p = |F^n(0)^2 + c|_p\] is bounded by $\max(|c|_p, r^2)\leq r$. Therefore, by induction, for all $|c|_p \leq 1$, $c$ is in the $p$-adic Mandelbrot set. For $|c|_p > 1$, $|F^n(0)|_p$ grows double exponentially with $n$, so any point outside the unit disk is not in the $p$-adic Mandelbrot set $\calm_p^2$.
\end{proof}
We want to study the more complex dynamics in the $p$-adic space by expanding the definition of $p$-adic Mandelbrot sets to a more general case. The reason why we only consider a single type of orbit $\{F^n(0)\}$ is that each function in the quadratic family has only one critical point at the origin, while higher-degree polynomials have multiple critical points and, therefore, more diverse types of dynamical and connectivity behaviors determined by Theorem \ref{thm:connectivity}. With this in mind, we can define the Mandelbrot set of higher-degree polynomials in $\C_p$. 

\begin{definition}
    For a vector $v := (c_1, c_2, ...,c_{d-1})\in\C_p^{d-1}$, consider the degree $d$ polynomial $F_v(z) = z^d + c_1z^{d-2} + c_2z^{d-3} + ...+c_{d-2}z+c_{d-1}$. The $p$-adic Mandelbrot set $\calm_p^d$ of degree $d$ can be defined as the set of $v$ such that all critical orbits of $F_v$ are bounded in $\C_p$.
\end{definition}

To know whether the structure of a given $p$-adic Mandelbrot set is simple, we can use the following theorem to conclude that some Mandelbrot sets share the same geometric property.
\begin{theorem}
        \label{thm:disk} \cite[Theorem 1]{MR3113229} 
    If $p\geq n$, then $\calm_p^n$ is the poly-disk \[\{(c_1, c_2,...,c_{n-1})\text{ where }|c_k|_p\leq1 \text{ for } 1\leq k\leq n-1 \}.\]
\end{theorem}
This theorem can also be used to explain why the Mandelbrot set of the quadratic family (where $n=2$) is a unit disk. Some calculations and examples in \cite[Page 2 and 3]{MR3113229} and \cite[Figure 1, 2, and 3]{MR3040666} also show that the complexity of some $\calm_p^n$ with $n>p$ can be comparable to its counterpart in the complex-dynamical setting.
\subsection{Examples} \cite[Example 2]{MR3113229}
Consider the function \[F:\C_2\rightarrow\C_2,\text{ }F(z) = z^3 - \frac{3}{4}z-\frac{3}{4}.\] By solving $F'(z) = 3z^2 - \frac{3}{4} = 0$, we find that the critical points of $F$ are $\pm\frac{1}{2}$. We can find out that the critical orbit at $\frac{1}{2}$ is bounded:
\[
\begin{tikzcd} \frac{1}{2} \arrow[r, mapsto, "F"] & -1 \arrow[r, mapsto, "F"] & -1. \end{tikzcd}
\]
We also know that the critical orbit at $-\frac{1}{2}$ is bounded:
\[
\begin{tikzcd}
-\frac{1}{2} \arrow[r, mapsto, "F"] & -\frac{1}{2}.
\end{tikzcd}
\]
Therefore, the pair $(-\frac{3}{4},-\frac{3}{4})\in\calm_2^3$. Since $|\frac{3}{4}|_2 = 4$, the Mandelbrot set $\calm_2^3$ is not a unit poly-disk. 

Similarly, we can define the $p$-adic analog of the filled Julia set in $\C_p$. 
\begin{definition}
    For a given function $F:\C_p\rightarrow\C_p$, we can define the $p$-adic filled Julia set $Q_p(F)$ as the set of points $z$ where the orbit of $z$ under $F$ is bounded. 
\end{definition}
The $p$-adic filled Julia set is a useful tool for studying the arithmetic dynamics of a given function in the $p$-adic field.
\section{$p$-adic alternated Julia sets}
In this section, we will extend the results from \cite{MR2567586}, \cite{MR3040666}, and \cite{MR3113229} to propose the definition of $p$-adic alternated filled Julia sets and some basic properties of them. Based on the previous discussion, we can formulate the definition of the orbits of two given alternated functions $F_1:\C_p\rightarrow\C_p$ and $F_2:\C_p\rightarrow\C_p.$
For simplicity, we start the discussion focusing on the system consisting of the quadratic family.
\begin{definition}
    For two given functions 
    \[
      \begin{tikzcd}[row sep=tiny]
        F_1: \C_p \arrow[r] & \C_p \\
        z \arrow[r, mapsto] & z^2 + c_1 \end{tikzcd}\]
\text{ and} 
        \[\begin{tikzcd}[row sep=tiny]
        F_2: \C_p \arrow[r] & \C_p \\
        z \arrow[r, mapsto] & z^2 + c_2 
      \end{tikzcd}\]
    
     where $c_1$ and $c_2$ are constants in $\C_p$, we define the orbit $P^{c_1c_2}_p(z_0)$ of a given starting point $z_0\in\C_p$ as the set $\{z_i|z_{2i} = F_2(z_{2i-1}), z_{2i+1} = F_1(z_{2i})\}$.
\end{definition}
We can also define the $p$-adic alternated Julia set in this manner.
\begin{definition}
    The $p$-adic alternated Julia set $K^{c_1c_2}_p$ of $F_1(x) = x^2+c_1, F_2(x) = x^2+c_2\in\C_p[x]$ is the set of points $z$ where its orbit $P^{c_1c_2}_p(z)$ is bounded under the $p$-adic distance metric. 
\end{definition}
% \begin{definition}
%     The $p$-adic Mandelbrot set of alternated Julia sets is the set of $(c_1, c_2)$ such that $K^{c_1c_2}_p$ is connected.
% \end{definition}
In contrast to the normal quadratic $p$-adic set that has only one critical orbit, the structure of the $p$-adic alternated Julia set is more complex. To describe the structure of the $p$-adic alternated Julia set, we have the following theorem.
\begin{theorem}
    \label{thm:p-adic boundedness}
    The even terms $z_{2i}$ of the orbit $P^{c_1c_2}_p(z_0)$ are bounded iff the odd terms $z_{2i+1}$ of the orbit are also bounded. 
\end{theorem}
\begin{proof}
    Suppose that for all $n\in\N$, $|z_{2n}|\leq r$ where $r$ is a non-negative real number. We have
    \begin{align*}
        |z_{2n+1}|_p &= |z_{2n}^2 + c_1|_p \\
                & \leq \max(|z_{2n}^2|_p, |c_1|_p) \\
                & \leq \max(r^2, |c_1|_p) \text{.}
    \end{align*}
    Therefore, all odd terms are bounded. If the set of all even terms $\{z_{2n}\}$ is unbounded,
\begin{align*}
    \sup |z_{2n+1}^2| &= \sup |z_{2n}^2 + c_1|_p \\
    &= \sup |z_{2n}^2|_p \to \infty.
\end{align*}
    As a result, the odd terms are also unbounded if the even terms are unbounded. Similarly, by symmetry we can show that if $|z_{2n-1}|$ is bounded, $|z_{2n}| =|z_{2n-1}^2 + c_2|$ is also bounded and if 
    $|z_{2n-1}|$ is unbounded then $|z_{2n}| =|z_{2n-1}^2 + c_2|$ should also be unbounded. 
\end{proof}
We can also formulate an auxiliary $p$-adic Julia set of the alternated system similar to what we have done in the complex case with the following theorem. 
\begin{theorem}
    \label{thm:general}
    Let $F_1, F_2 \in \C_p[x]$ and $F = F_2\circ F_1$. Then the $p$-adic alternated filled Julia set of $F_1$ and $F_2$ is given by $Q_p(F)$.
\end{theorem}
\begin{proof}
    The proof follows the same structure as the argument used in Theorem~\ref{thm:p-adic boundedness}. In particular, to show that the odd and even terms of an orbit share the same boundedness properties, we analyze the behavior of the terms using the $p$-adic absolute values. Since the coefficients in both $F_1$ and $F_2$ have finite $p$-adic absolute values and $\deg(F) < \infty$, the boundedness properties carry over from one subsequence to the other. 
\end{proof}

Based on Theorem~\ref{thm:general} and Theorem~\ref{thm:disk}, we can have the following Theorem.
\begin{theorem}    
    \label{thm:p-adic connectivity}
    Let $F_1, F_2 \in \C_p[x]$ and $F = F_2\circ F_1$. If $\deg(F) \leq p$ and all coefficients of $F$ have $p$-adic absolute values less than or equal to $1$, then the $p$-adic alternated Julia set formed by alternating between $F_1$ and $F_2$ is connected.
\end{theorem}
\begin{proof}
    If $\deg(F) \leq p$, then by Theorem~\ref{thm:disk}, we know that $\calm_p^{\deg(F)}$ is a unit poly-disk. Therefore, if all coefficients of $F$ have $p$-adic absolute values less than or equal to $1$, the coefficient vector of $F$ lies within $\calm_p^{\deg(F)}$. This implies that the $p$-adic Julia set $Q_p(F)$ of $F$ is connected. Since we know that $Q_p(F)$ is the $p$-adic alternated Julia set of $F_1$ and $F_2$ by Theorem~\ref{thm:general}, we can conclude that the $p$-adic alternated Julia set is connected.
\end{proof}
With the results established in the previous theorems, we can infer the connectivity of the $p$-adic alternated Julia set of functions in the quadratic family by examining the auxiliary quartic polynomial associated with the alternated dynamics and analyzing its parameters with respect to the $p$-adic Mandelbrot set $\calm_p^4$. According to Theorem~\ref{thm:disk}, when $p\geq 4$, $\calm_p^4$ forms a poly-disk, which allows us to easily determine whether the corresponding $p$-adic Julia set is connected. In contrast, for $p\leq3$, the structure of $\calm_p^4$ becomes highly intricate, and the connectivity of the $p$-adic Julia set can no longer be reliably inferred without computational assistance.

Given a pair of parameters $c_1$ and $c_2$, we can construct the associated auxiliary quartic polynomial and use it along with the geometry of \(\calm_p^4\) to determine the connectivity of the \(p\)-adic alternated Julia set. This method extends to arbitrary polynomial functions \(F_1\) and \(F_2\), where the connectivity of the \(p\)-adic Julia set of the composed map \(F = F_2\circ F_1\) can be understood by checking the coefficients of \(F\) and comparing \(p\) with \(\deg(F)\) based on Theorem~\ref{thm:p-adic connectivity}.

\subsection{Example}
Consider the following two functions:
    \[
      \begin{tikzcd}[row sep=tiny]
        F_1: \C_p \arrow[r] & \C_p \\
        z \arrow[r, mapsto] & z^2 - \frac{1}{4} \end{tikzcd}\]
\text{ and} 
        \[\begin{tikzcd}[row sep=tiny]
        F_2: \C_p \arrow[r] & \C_p \\
        z \arrow[r, mapsto] & z^2 + \frac{1}{2}. 
      \end{tikzcd}\]

Then, the composition \( F = F_2 \circ F_1 \) is given by:
\begin{align*}
  F(z) &= (F_2 \circ F_1)(z) \\
       &= \left(z^2 - \frac{1}{4}\right)^2 + \frac{1}{2} \\
       &= z^4 - \frac{1}{2} z^2 + \frac{9}{16}, \\
  F'(z) &= 4z^3 - z.
\end{align*}

The critical points of \( F \) are \( \pm \frac{1}{2} \) and \( 0 \). Note that since \(\deg(F) = 4 > 2 \), we cannot determine the connectivity of \( Q_2(F) \) solely by inspecting the coefficients. Since \( \frac{1}{2} \) is a fixed point and $F(-\frac{1}{2}) = \frac{1}{2}$, we only need to consider the critical point at the origin.

For \( p = 2 \),
\[
  |F^n(0)|_2 = 2^{4^n} \to \infty,
\]
so \( Q_2(F) \) is not connected.

For all \( p \geq 5 \), we have \( \deg(F) = 4 \leq p \), and all coefficients of \( F \) have \( p \)-adic absolute values equal to 1. Therefore, by Theorem~\ref{thm:p-adic connectivity}, \( Q_p(F) \) is connected.

In essence, this provides a practical criterion: By reducing an alternated dynamic to a single auxiliary polynomial, we can leverage known properties of $p$-adic Mandelbrot sets to deduce the connectivity of $p$-adic alternated Julia sets.
%%%%%%%%%%%%%%%%%%%%%%%%%%%% References %%%%%%%%%%%%%%%%%%%%%%%%%%%%%%

\bibliography{Review}{}
\bibliographystyle{alpha}

\end{document}